\documentclass{amsart}

\newcommand{\be}{\begin{equation}}               
\newcommand{\ee}{\end{equation}}                 
\newcommand{\stac}{\stackrel}                    
\newcommand{\noin}{\noindent}                    

 \newcommand{\ignore}[1]{}{}

 \newcommand{\la}{\lambda}
 
 \newcommand{\beq}{\begin{eqnarray*}}
 \newcommand{\eeq}{\end{eqnarray*}}
 \newcommand{\beqn}{\begin{eqnarray}}
 \newcommand{\eeqn}{\end{eqnarray}}
 \newcommand{\ta}{\tau}
 \newcommand{\ra}{\rightarrow}
 \newcommand{\mb}{\mbox}
 
 \newcommand{\ep}{\epsilon}
 \newcommand{\al}{\alpha}
 
 \newcommand{\bi}{\begin{itemize}}
 \newcommand{\ei}{\end{itemize}}
 
 \newcommand{\La}{\Lambda}

 \newcommand{\ssb}{\scriptstyle \footnotesize 
                  \begin{array}{c}}
 \newcommand{\esb}{\end{array}}

 \newcommand{\nn}{\nonumber}

 \newcommand{\lbl}{\label}

 \newcommand{\eq}[1]{$(\ref{#1})$}

 \renewcommand{\theequation}{\arabic{section}.\arabic{equation}}
 \newcommand{\de}{\delta}
 \newcommand{\De}{\Delta}
 \newcommand{\ga}{\gamma}
 \newcommand{\Ga}{\Gamma}
 \newcommand{\si}{\sigma}

\newtheorem{theorem}{Theorem}[section]
\newtheorem{lemma}[theorem]{Lemma}

\theoremstyle{definition}

\theoremstyle{remark}

\numberwithin{equation}{section}

\newcommand{\bea}{\begin{eqnarray}}              
\newcommand{\eea}{\end{eqnarray}}

\begin{document}

\title{ Local Semicircle law and Gaussian fluctuation for Hermite $\beta$ ensemble$^*$}

\author{Zhigang Bao and Zhonggen Su}
\address{Department of Mathematics, Zhejiang University, Hangzhou, Zhejiang 310027, P.R. China}
\thanks{$*$Partially supported  by NSFC grant No.11071213 and ZJNSF grant No.R6090034 }
\email{maomie2007@gmail.com}
\email{suzhonggen@zju.eud.cn}


\date{January 20, 2010} 



\begin{abstract}
Let $\beta>0$ and consider an $n$-point process $\lambda_1, \lambda_2,\cdots, \lambda_n$ from Hermite $\beta$ ensemble on the real line $\mathbb{R}$. Dumitriu and Edelman discovered a tri-diagonal matrix model and established the global Wigner  semicircle law for normalized empirical measures. In this paper we prove that the average number of states in a small interval in the bulk converges in probability  when the length of the interval is  larger than $\sqrt {\log n}$, i.e., local semicircle law holds. And the number of positive states in $(0,\infty)$ is proved to fluctuate normally around its mean $n/2$ with variance like $\log n/\pi^2\beta$. The proofs rely largely on the way invented by Valk$\acute{\mbox {o}}$ and Vir$\acute{\mbox {a}}$g of counting  states in any interval  and the classical martingale argument.

\end{abstract}

\maketitle

\pagestyle{myheadings} \markright{ Local Semicircle law and Gaussian fluctuation for H$\beta$E  }

\section{Introduction }
\setcounter{equation}{0}

Consider an $n$-point process $\lambda_1, \lambda_2,\cdots, \lambda_n $ on the real line $\mathbb{R}$ with the following joint probability density function
\begin{eqnarray}
p_n(x_1, \cdots, x_n)=\frac{1}{Z_{n,\beta}}\prod_{i<j}|x_i-x_j|^\beta\prod_{j=1}^ne^{-\frac{\beta}{4}x^2_j}, \quad x_1,\cdots, x_n\in \mathbb{R},\label{eq.1.1}
\end{eqnarray}
where $\beta>0$ is a model parameter and $Z_{n,\beta}$  the normalization constant. This was first introduced by Dyson \cite{Dyson1962} in the study of Coulomb lattice gas in the early sixties, and  is usually referred to as Hermite $\beta$ ensemble (H$\beta$E) in the literature.  The formula (\ref{eq.1.1}) can be rewritten in a more familiar form to physicists:
\begin{eqnarray}
p_n(x_1, \cdots, x_n) \propto   e^{-\beta{  H}_n(x_1,\cdots, x_n)},
\quad x_1,\cdots, x_n\in \mathbb{R}, \nonumber\label{eq.1.2}
\end{eqnarray}
where ${  H}_n(x_1,\cdots, x_n)=\frac 14\sum_{j=1}^nx^2_j -\frac 12\sum_{i\neq j} \log|x_i-x_j|$ is a Hamiltonian system.

Note that $\beta $ stands for inverse temperature, the quadratic function part means the points fall independently in the real line with normal law, while the extra logarithmic part indicates the points repel each other.

 The special cases $\beta=1,2,4$  correspond to Gaussian Orthogonal Ensemble, Gaussian Unitary Ensemble and Gaussian Symplectic Ensemble respectively, which  are one of most studied objects in random matrix theory. The reader is referred to  a classical book Mehta \cite{M2004} for more background.

In this paper, we are mainly interested in large $n$ asymptotic behaviors of Hermite $\beta $ ensembles with general $\beta>0$.  In particular, we will investigate the local behavior of points in a very small interval in the bulk and the fluctuation of the number of points in an half-infinite interval  around its mean.

To state our main results, let us first introduce some notations and review recent relevant progress about H$\beta$E. A remarkable breakthrough was made by Dumitriu and Edelman \cite{DE2002}, in which they discovered a tri-diagonal matrix model representation for H$\beta$E, see Section 2 below for matrix model.   There have since then been rapid development in the study of H$\beta$E within past few years. Dumitriu and Edelman \cite{DE2002} made use of such a tri-diagonal matrix model and moment methods to prove the following fundamental law of large number for empirical measures.  Let
\begin{eqnarray}
\rho_{sc}(x)= \frac{1}{2\pi}\sqrt{4-x^2},\quad |x|\le 2, \label{eq.1.3}
\end{eqnarray}
then it follows for any fixed $a<b$
\begin{eqnarray}
\frac{1}{n}\sum_{i=1}^n\mathbf{1}_{(a< \frac{\lambda_i}{\sqrt n}\le b)}\stackrel P\longrightarrow\int_{a}^b\rho_{sc}(x)dx. \label{eq.1.4}
\end{eqnarray}
This is so-called global Wigner  semicircle law since it was first discovered by Wigner \cite{W1955}.

Let $\lambda_{(1)}\le \lambda_{(2)}\le \cdots\le \lambda_{(n)}$ be the ordered arrangement of points in the real line $\mathbb{R}$. Ram$\acute{\mbox i}$rez,  Rider and  Vir$\acute{\mbox {a}}$g \cite{RRV2007} established via variational analysis  the $\beta$ type of Tracy-Widom law  for the rightmost endpoints as follows. For any fixed integer $k\ge 0$, \begin{eqnarray}
 n^{1/6}(\lambda_{(n-k)}-2\sqrt n)\stackrel d\longrightarrow -\Lambda_{(k)},\label{eq.1.5}
\end{eqnarray}
where $\Lambda_{(k)}$ is the $k+1$-lowest eigenvalue of stochastic Airy operator. We remark that the limiting distribution in the righthand side of (\ref{eq.1.5}) can be expressed explicitly in terms of Painlev$\acute{\mbox e}$ II equation in  special cases $\beta=1, 2,4$, while there is not a suitable computable expression for general $\beta>0$ yet.

One can readily see from (\ref{eq.1.4}) and (\ref{eq.1.5}) that the spacings between points in the bulk are asymptotically the same order as $\frac 1{\sqrt{n}}$, and the spacings  near the edge are asymptotically as large as $\frac 1{n^{1/6}}$.

Only recently did Valk$\acute{\mbox {o}}$ and  Vir$\acute{\mbox {a}}$g \cite{VV2009} make a new wonderful contribution to the weak convergence of random H$\beta$E point processes. They counted the numbers of suitably scaled points in any fixed interval  like $(0, \lambda)$ or $(-\lambda, 0)  (\lambda>0)$ and proved these numbers converges weakly to corresponding numbers of $\mbox{Sine}_\beta$ point process.  The $\mbox{Sine}_\beta$ point process is closely related to Brownian carousel and reduces to the well-known sine point process with kernel $K(x,y)=\frac{\sin \pi(x-y)}{\pi(x-y)}$ when $\beta=2$. One of key techniques in their argument is to use again the Dumitriu and Edelman tri-diagonal matrix representation and to find a sufficient and necessary condition for a real number to be its eigenvalue in terms of phase evolution of ratios of consecutive coefficients of  eigenvectors.

Now we are ready to state our main results.
\begin{theorem} \label{th.1.1} Assume $ t_n, n\ge 1$ is  a sequence of real numbers  such that
$$ \frac{t_n}{\sqrt{\log n}}\rightarrow\infty ,\qquad \frac{t_n}{\sqrt n}\rightarrow 0\quad {\mbox{ as }} n\rightarrow\infty.$$
 Then for any  $-2<x<2$ and any $\varepsilon>0$,
\begin{eqnarray}
P\left(\left|\frac{N_n\left(x\sqrt n, x\sqrt n+\frac{t_n}{\sqrt n}\right)}{t_n}-\rho_{sc}(x)\right|>\varepsilon\right)\longrightarrow 0, \nonumber\label{eq.1.6}
\end{eqnarray}
where $N_n(a, b)$ denotes the number of  the points $\lambda_i$ in (\ref{eq.1.1}) in the interval $ (a, b)$ and $\rho_{sc}(x)$ is as in (\ref{eq.1.3}).
\end{theorem}
In contrast to the global semicircle law in (\ref{eq.1.4}), Theorem \ref{th.1.1} characterizes the density of states in a small interval  around $x\sqrt n$. So, it is called the local semicircle law. This was first studied in Erd$\ddot{\mbox o}$s et al. \cite{ESY2008} and was then improved to almost the optimal scale in   \cite{ESY2009} and \cite{ESY2010} in the context of Wigner random matrices.  We remark that Erd$\ddot{\mbox o}$s et al. \cite{ESY2008, ESY2009, ESY2010}   do not only prove convergence in probability of  the density of states  to $\rho_{sc}$, but also obtain an exponential decay tail estimate under certain exponential integrability conditions.

Our next result is concerned with the fluctuation of the number of points in an half-infinite interval around its mean.
\begin{theorem} \label{th.1.2} Consider the point process $\lambda_i$ in  (\ref{eq.1.1}), and  let
\begin{eqnarray}
N_n(0,\infty)=\sharp\left \{1\le i\le n;\quad  0\le \lambda_i< \infty\right\}. \nonumber\label{eq.1.7}
\end{eqnarray}
Then we have
\begin{eqnarray}
  \frac{N_n(0,\infty) -\frac n2}{\frac{1}{2\pi}\sqrt { \log n}}\stackrel d\longrightarrow \mathcal{N}\left(0,\frac 4{\beta} \right).\label{eq.1.8}
\end{eqnarray}
\end{theorem}
The central limit theorem like Theorem \ref{th.1.2} has  been  known for determinantal point processes since Costin and Lebowitz \cite{CL1995}, which first studied a specific determinantal point process with sine kernel.  Soshnikov \cite{So2000} then investigated general random point fields and proved the following basic central limit theorem.

Let ${ \mathcal{X}}_n$ be a sequence of determinantal point processes in $\mathbb{R}$ with kernel $K_n$. Let $I_n$ be a sequence of Borel sets in $\mathbb{R}$ such that $\sharp I_n $ (the number of points of $  \mathcal{ X} _n$ in $I_n$)  is finite a.s. and $Var(\sharp I_n )\rightarrow \infty$. Then
\begin{eqnarray}
\frac{\sharp I_n -E\sharp I_n }{\sqrt{Var(\sharp I_n )}}\stackrel d\longrightarrow \mathcal{N}(0,1).  \nonumber \label{eq.1.9}
\end{eqnarray}
Recently did Hough et al. \cite{HKPV2006} give a conceptual probabilistic proof. The fact that the correlation functions have determinantal structure plays a significant role in all their arguments.  H$\beta$E is  obviously no longer a determinantal point process unless $\beta=2$. It would be interesting to investigate if the central limit theorem holds for general $\beta$ ensembles and even for Wigner matrices.

In Theorem \ref{th.1.2} we discussed  only the number of positive eigenvalues. We conjecture, however, an analog holds for the  number $N_n [x\sqrt n, \infty) $ of points in $[x\sqrt n, \infty), -2<x<2$.  The number $N_n (0, \infty) $ is called the index and is a key object of interest to physicists.  Cavagna et al. \cite{CGG2000} calculated the distribution of the index  for GOE by means of the replica method  and obtained Gaussian distribution with asymptotic variance like $ {\log n}/{\pi^2}$. Majumdar et al. \cite{MNSV2009} and \cite{MNSV2010}  further computed analytically the probability distribution of the number $N_n(0, \infty)$ of positive eigenvalues for H$\beta$E ($\beta=1,2,4$) using the partition function  and saddle point analysis. They computed the variance  $ {\log n}/{\pi^2 \beta }+O(1)$, which agrees with the corresponding variance in (\ref{eq.1.8}), while they thought the distribution is not strictly Gaussian due to an usual logarithmic singularity in the rate function.  But  the  variance like $\log n$  is actually typical in the central limit theorem  for the numbers  in random matrix theory.

The rest part of the paper  will focus on proving Theorems \ref{th.1.1} and  \ref{th.1.2}. The proofs rely largely on the new phase evolution of ratios of consecutive coefficients of eigenvectors invented by Valk$\acute{\mbox {o}}$ and  Vir$\acute{\mbox {a}}$g \cite{VV2009}. For reader's convenience we shall in Section 2 introduce some necessary notations and give a brief description of Valk$\acute{\mbox {o}}$ and  Vir$\acute{\mbox {a}}$g's basic identity for the number of the states in any interval (see (\ref{eq.2.10}) and $(v^\prime)$ below, see also (v) of Proposition 18 in \cite{VV2009}).  A key  point is that the difference $\Delta \varphi_{l,\lambda}$ forms an array of Markov chain so that the classical martingale argument is applicable. Section 3 contains technical estimates for variances and verification of the martingale type Linbeberg condition.

\section{Valk$\acute{\mbox {o}}$ and Vir$\acute{\mbox {a}}$g's phase evolution}

Consider the following random tri-diagonal matrix
\begin{eqnarray}
H_n^\beta= \frac{1}{\sqrt \beta}
\left(
\begin{array}{ccccc}
a_0 & b_0&0&\cdots&0\\
b_0&a_1&b_1&\cdots&0 \\
0&b_1&a_2&\cdots&0\\
\vdots&\vdots&\cdots&\vdots&\vdots\\
0&\cdots&\cdots&\cdots&a_{n-1}
\end{array}
\right),\label{eq.2.1}
\end{eqnarray}
where $a_0, a_1, \cdots, a_{n-1}$ are independent normal random variables with $ a_i\sim N(0,2)$, $b_0, b_1, \cdots, b_{n-2}$ are independent chi random variables with $ b_i\sim \chi_{(n-i-1)\beta}$;   and the $a_i$'s are independent of the $b_i$'s.

A remarkable contribution to the study of H$\beta$E due to  Dumitriu and Edelman \cite{DE2002} is that the eigenvalues of $H_n^\beta$ have (\ref{eq.1.1}) as their joint probability density function. Thus we need only to consider the eigenvalues of $H_n^\beta$, denoted for simplicity  still by  $\lambda_1, \lambda_2,\cdots, \lambda_n$. In the recent work of  Valk$\acute{\mbox {o}}$ and Vir$\acute{\mbox {a}}$g,  (\ref{eq.2.1}) is used to derive a recurrence equation for a real number $\Lambda$ to be an eigenvalue, which in turn yields a certain evolution relation for eigenvectors.
The specific relation is as follows.

Let $s_j=\sqrt{n-j-\frac 12}$.  Define
 \begin{eqnarray}
 D_n=\left(
 \begin{array}{ccccc}
 d_{11}&0&0&\cdots&0\\
 0&d_{22}&0&\cdots&0\\
 \vdots&\vdots&\vdots&\cdots&\vdots\\
 0&0&0&\cdots&d_{nn}
 \end{array}
 \right),\nonumber\label{eq.2.2}
\end{eqnarray}
 where
 $$
 d_{11}=1, \quad  d_{ii}=\frac{b_{i-2 }}{s_{i-1}}d_{i-1, i-1}, \quad 2\le i\le n.
 $$
Let
$$X_i=\frac{a_i}{\sqrt \beta},\qquad 0\le i\le n-1$$
and
$$Y_i=\frac{b^2_{i}}{\beta  {s_{i+1}}}-s_i,\qquad 0\le i\le n-2. $$
Then
 \begin{eqnarray}
 D_n^{-1}H_n^\beta D_n=\left(
 \begin{array}{ccccc}
  X_0&s_0+Y_0&0&\cdots&0\\
  s_1&X_1&s_1+Y_1&\cdots&0\\
  0&s_2&X_2&\cdots&0\\
  \vdots&\vdots&\vdots&\cdots&\vdots\\
  0&0&0&\cdots&X_{n-1}
 \end{array}
 \right)\nonumber\label{eq.2.3}
\end{eqnarray}
obviously have the same eigenvalues as $H_n^\beta$.

Note that there is a significant difference between these two matrices.  The rows between $D_n^{-1}H_n^\beta D_n$ are independent of each other, while  $H_n^\beta$ is symmetric so that the rows are not independent.

Assume that $\Lambda$ is an eigenvalue of  $D_n^{-1}H_n^\beta D_n$, then by definition there exists a nonzero  eigenvector $\mathbf{\nu}^\tau=(\nu_1,\nu_2,\cdots, \nu_n)$  such that
$$
D_n^{-1}H_n^\beta D_n  \mathbf{\nu}= \Lambda\mathbf{\nu}.
$$
Without loss of generality, we can assume $\nu_1=1$. Thus, $\Lambda$ is an eigenvalue if and only if there exists an eigenvector $\mathbf{\nu}^\tau=(1,\nu_2,\cdots, \nu_n)$ such that
\begin{eqnarray}
  \left(
 \begin{array}{ccccc}
  X_0&s_0+Y_0&0&\cdots&0\\
  s_1&X_1&s_1+Y_1&\cdots&0\\
  0&s_2&X_2&\cdots&0\\
  \vdots&\vdots&\vdots&\cdots&\vdots\\
  0&0&0&\cdots&X_{n-1}
 \end{array}
 \right)\left(
 \begin{array}{c}
 1\\
  \nu_2\\
  \nu_3\\
  \vdots \\
  \nu_n
 \end{array}
 \right)=\Lambda\left(
 \begin{array}{c}
 1\\
  \nu_2\\
  \nu_3\\
  \vdots \\
  \nu_n
 \end{array}
 \right).\nonumber\label{eq.2.4}
\end{eqnarray}
It can in turn be equivalently rewritten into
\begin{eqnarray}
  \left(
 \begin{array}{ccccccc}
  1&X_0&s_0+Y_0&0&\cdots&0&0\\
  0&s_1&X_1&s_1+Y_1&\cdots&0&0\\
  0&0&s_2&X_2&\cdots&0&0\\
  \vdots&\vdots&\vdots&\vdots&\cdots&\vdots&\vdots\\
  0&0&0&0&\cdots&X_{n-1}&1
 \end{array}
 \right)\left(
 \begin{array}{c}
 0\\
 1\\
  \nu_2\\
  \nu_3\\
  \vdots \\
  \nu_n\\
  0
 \end{array}
 \right)=\Lambda\left(
 \begin{array}{c}
 1\\
  \nu_2\\
  \nu_3\\
  \vdots \\
  \nu_n
 \end{array}
 \right).\nn
\end{eqnarray}
Let $\nu_0=0, \nu_{n+1}=1$ and define
$
 r_l=\frac{\nu_{l+1}}{\nu_l}, 0\le  l\le n.
$
Thus we have the following necessary and sufficient condition for $\Lambda$ to be an eigenvalue in terms of  evolution:
 \begin{eqnarray}
 \qquad\infty=r_0,\quad  r_{l+1}=\frac{1}{1+\frac{Y_l}{s_l}}\left(-\frac{1}{r_l}+\frac{\Lambda-X_l}{s_l}\right),\quad 0\le  l\le n-2, \quad  r_n=0.\label{eq.2.5}
\end{eqnarray}
Since the $(X_l, Y_l)$'s are independent, then $r_0, r_1, \cdots, r_{n-1}, r_n$ forms a Markov chain with $\infty$ as initial state and $0$ as destination  state, and the next state $r_{l+1}$ given a present state $r_l$ will be attained through a random fractional linear transform.

Next we turn to the description of the phase evolution. Let $\mathbb{H}$ denote the upper half plane, $\mathbb{U}$ the Poincar$\acute{\mbox e}$ disk model, define the bijection
$$
\mathbf{U}: \mathbb{\bar{H}}\rightarrow \mathbb{\bar{U}},\quad z\rightarrow \frac {i-z}{i+z},
$$
which is also a bijection of the boundary. As $r$ moves on the boundary $\partial \mathbb{H}=\mathbb{R}\cup \{\infty\}$, its image under $\mathbb{U}$ will move along $\partial\mathbb{U}$.

In order to follow the number of times this image circles $\mathbb{U}$, we need to extend the action from $\partial\mathbb{U}$ to its universal cover, $\mathbb{R}'=\mathbb{R}$, where the prime is used to distinguish this from $\partial \mathbb{H}$. For an action  $\mathbf{T}$ on $\mathbb{R}'$, the three actions are denoted by
$$
\mathbb{\bar{H}}\rightarrow \mathbb{\bar{H}}: z\rightarrow z_\cdot \mathbf{T},\quad  \mathbb{\bar{U}}\rightarrow \mathbb{\bar{U}}: z\rightarrow z_\circ \mathbf{T}, \quad \mathbb{R}'\rightarrow \mathbb{R}': z\rightarrow z_\ast \mathbf{T}.
 $$
Let $\mathbf{Q}(\alpha)$ denote the rotation by $\alpha$ in $\mathbb{U}$ about $0$, i.e.,
$$
\varphi_*\mathbf{Q}(\alpha) =\varphi+\alpha.
$$
For $a, b\in \mathbb{R}$ let $\mathbf{A}(a, b)$ be the affine map $z\rightarrow a(z+b)$ in $\mathbb{H}$, and it  acts on $\mathbb{R}'$ as follows
\begin{eqnarray}
\varphi_*\mathbf{A}(a, b):= \mbox{Arg}(\mathbf{U}(\mathbf{A}(a, b)[\mathbf{U}^{-1}(e^{i\varphi})])), \label{eq.2.5.-1}\nonumber
\end{eqnarray}
where the argument is specified by the convention of fixing $\pi$ under the action of $\mathbf{A}(a, b)$ and $|\varphi_*\mathbf{A}(a, b)-\varphi|<2\pi$. In other words, we can redefine $\varphi_*\mathbf{A}(a, b)$ by
\begin{eqnarray}
\varphi_*\mathbf{A}(a, b):= \varphi+ ash(\mathbf{A}(a, b),-1, e^{i\varphi}), \label{eq.2.5.-2}\nonumber
\end{eqnarray}
where
 \begin{eqnarray}
\quad ash(\mathbf{A}(a, b),v,  w)=\mbox{Arg}_{[0,2\pi)}\Bigl(\frac{\mathbf{U}(\mathbf{A}(a, b)[\mathbf{U}^{-1}(w)])}{\mathbf{U}(\mathbf{A}(a, b)[\mathbf{U}^{-1}(v)])}\Bigr)- \mbox{Arg}_{[0,2\pi)}\Bigl(\frac wv\Bigr)\label{eq.2.5.-3}\nonumber
\end{eqnarray}
for all $w,v\in e^{i\mathbb{R}}$. Furthermore, define
$$
\mathbf{W}_l=\mathbf{A}\left(\frac{1}{1+\frac{Y_l}{s_l}}, -\frac{X_l}{s_l} \right), \quad \mathbf{R}_{l, \Lambda}=\mathbf{Q}(\pi)\mathbf{A}\left(1, \frac{\Lambda}{s_l}\right)\mathbf{W}_l,\quad 0\le l\le n-1.
$$
With this notation, the evolution of  $r$ in (\ref{eq.2.5}) becomes
\begin{eqnarray}
r_{l+1}={r_l}_\cdot \mathbf{R}_{l, \Lambda}, \quad 0\le l\le n-1\nonumber \label{eq.2.6}
\end{eqnarray}
and $\lambda$ is an eigenvalue if and only if $\infty_\cdot \mathbf{R}_{0, \Lambda}\cdots\mathbf{R}_{n-1, \Lambda}=0$. For $0\le l\le n$ define
\begin{eqnarray}
\hat{\varphi}_{l, \Lambda}=\pi_*\mathbf{R}_{0, \Lambda}\cdots\mathbf{R}_{l-1, \Lambda}, \quad \hat{\varphi}^\odot_{l, \Lambda}=0_*\mathbf{R}_{n-1, \Lambda}^{-1}\cdots\mathbf{R}^{-1}_{l, \Lambda},\nn\label{eq.2.7}
\end{eqnarray}
then
\begin{eqnarray}
\hat{\varphi}_{l, \Lambda} = \hat{\varphi}^\odot_{l, \Lambda} \quad \mbox{mod }\, 2\pi.\nn\label{eq.2.8}
\end{eqnarray}
Now we state there exist functions $\hat{\varphi}, \hat{\varphi}^\odot: \{0,1,\cdots, n\}\times \mathbb{R}\rightarrow \mathbb{R}$ satisfying the following properties:

\noindent
($i$) $ r_{l, \Lambda \cdot} \mathbf{U}=e^{i\hat{\varphi}_{l, \Lambda} }$;

\noindent
($ii$) $\hat{\varphi}_{0, \Lambda} =\pi,\quad \hat{\varphi}^\odot_{n, \Lambda}=0 $;

\noindent
($iii$) For each $0<l\le n$, $\hat{\varphi}_{l, \Lambda}$ is an analytic and strictly increasing in $\Lambda$. For $0\le l<n$, $\hat{\varphi}^\odot_{l, \Lambda}$ is analytic and strictly decreasing in $\Lambda$;

\noindent
($iv$) For any $0\le l\le n$, $\Lambda$ is an eigenvalue of $H_n^\beta$ if and only if $\hat{\varphi}_{l, \Lambda}-\hat{\varphi}^{\odot}_{l, \Lambda}\in 2\pi \mathbb{Z}$.

Fix $-2<x<2$ and $n_0=n(1-\frac{x^2}{4})-\frac 12$. Let $\Lambda=x\sqrt n+\frac{\lambda}{2\sqrt {n_0}}$ and recycle the notation $r_{l, \lambda}$, $\hat{\varphi}_{l, \lambda}$, $\hat{\varphi}^{\odot}_{l, \lambda}$ for the quantities $r_{l, \Lambda}$, $\hat{\varphi}_{l, \Lambda}$, $\hat{\varphi}^{\odot}_{l, \Lambda}$.

Note that there is a macroscopic term $\mathbf{Q}(\pi)\mathbf{A}(1, \frac{\Lambda}{s_l})$ in the evolution operator $\mathbf{R}_{l, \Lambda}$. So the phase function $\hat{\varphi}_{l, \Lambda}$ exhibits fast  oscillation in $l$.

Let
$$
\mathbf{J}_l=\mathbf{Q}(\pi)\mathbf{A}\left(1, \frac{x\sqrt n}{s_l}\right)
$$
and
\begin{eqnarray}
\rho_l=\sqrt{\frac{nx^2/4}{nx^2/4+n_0-l}}+i\sqrt{\frac{n_0-l}{nx^2/4+n_0-l}}.\nn
\end{eqnarray}
Thus $\mathbf{J}_l$ is a rotation since $\rho_{l \cdot}\mathbf{J}_l=\rho_l$ and $\rho_l$ is unique in $\bar{\mathbb{H}}$. We separate $\mathbf{J}_l$ from the evolution operator $\mathbf{R}$ to get
$$
\mathbf{R}_{l, \lambda}=\mathbf{J}_l \mathbf{L}_{l, \lambda}\mathbf{W}_l, \quad
\mathbf{L}_{l, \lambda}=\mathbf{A}\left(1, \frac{\lambda}{2\sqrt {n_0}s_l}\right).
$$
Note that for any finite $\lambda$, $\mathbf{L}_{l, \lambda}$ and $\mathbf{W}_l$ becomes infintesimal in the $n\rightarrow\infty$ limit while $\mathbf{J}_l$ does not.

Let $$\mathbf{T}_l= \mathbf{A}\left(\frac{1}{\mbox{Im} \rho_l }, -\mbox{Re}  \rho_l \right),$$
then
$$
\mathbf{J}_l=\mathbf{Q}(-2\mbox{Arg}(\rho_l))^{\mathbf{T}_l^{-1}},
$$
where $A^B=B^{-1}AB$. Let
$$
\mathbf{Q}_l=\mathbf{Q}(2\mbox{Arg}(\rho_0))\ldots  \mathbf{Q}(2\mbox{Arg}(\rho_l))
$$
and
\begin{eqnarray}
{\varphi}_{l, \lambda}=\hat{\varphi}_{l, \lambda *}\mathbf{T}_l\mathbf{Q}_{l-1},\quad {\varphi}^\odot_{l, \lambda}=\hat{\varphi}^\odot_{l, \lambda *}\mathbf{T}_l\mathbf{Q}_{l-1}, \label{eq.2.10}
\end{eqnarray}
then  it is easy to see the following  properties hold: for every $0<l\le n_0$

\noindent
($i^\prime$) ${\varphi}_{0, \lambda}=\pi$;

\noindent
($ii^\prime$) ${\varphi}_{l, \lambda}$ and $-{\varphi}^\odot_{l, \lambda}$ are analytic and strictly increasing in $\lambda$ and are also independent;

\noindent
($iii^\prime$) with $\mathbf{S}_{l, \lambda}=\mathbf{T}_l^{-1}\mathbf{L}_{l,\lambda} \mathbf{W}_l \mathbf{T}_{l+1}$ and $\eta_l=\rho_0^2\rho_1^2\cdots \rho_l^2$, we have
\begin{eqnarray}
\Delta {\varphi}_{l, \lambda}:={\varphi}_{l+1, \lambda}-{\varphi}_{l, \lambda}=ash(\mathbf{S}_{l, \lambda}, -1, e^{i{\varphi}_{l, \lambda}}\bar{\eta}_l); \label{eq.2.10-1}
\end{eqnarray}

\noindent
($iv^\prime$) $\hat{{\varphi}}_{l, \lambda}={\varphi}_{l, \lambda *}\mathbf{Q}_{l-1}^{-1}\mathbf{T}_{l}^{-1}$;

\noindent
($v^\prime$) for any $\lambda<\lambda'$ we have a.s.
$$N_{n}\left(x\sqrt n+\frac{\lambda}{2\sqrt {n_0}}, \quad x\sqrt n+\frac{\lambda'}{2\sqrt {n_0}}\right)=\sharp(({\varphi}_{l, \lambda}-{\varphi}_{l, \lambda}^\odot,\quad {\varphi}_{l, \lambda'}-{\varphi}_{l, \lambda'}^\odot]\cap 2\pi \mathbb{Z}).$$

The difference $\Delta {\varphi}_{l, \lambda}$ in (\ref{eq.2.10-1}) can be estimated as follows. Let
\begin{eqnarray}
Z_{l, \lambda}&=&i_{\cdot}\mathbf{S}_{l, \lambda}^{-1}-i\nonumber\\
&=& i_{\cdot}\mathbf{T}_{l+1}^{-1}(\mathbf{L}_{l,\lambda} \mathbf{W}_l)^{-1} \mathbf{T}_{l} -i\nonumber\\
&=& v_{l, \lambda}+V_l, \label{eq.2.10-2}
\end{eqnarray}
where
$$
v_{l, \lambda}=-\frac{\lambda}{2\sqrt{n_0}\sqrt{n_0-l}}+\frac{\rho_{l+1}-\rho_l}{\mbox{Im} \rho_l },\quad V_l=\frac{X_l+\rho_{l+1}Y_l}{\sqrt{n_0-l}}.
$$
Then by (72) of \cite{VV2009}, it follows
\begin{eqnarray}
 \Delta {\varphi}_{l, \lambda}&=& ash(\mathbf{S}_{l, \lambda}, -1, e^{i{\varphi}_{l, \lambda}}\bar{\eta_l})\nonumber\\
 &=& \mbox{Re}\left[-(1+ {e^{-i{\varphi}_{l, \lambda}}}\eta_l)Z_{l,\lambda}-\frac{i(1+ {e^{-i{\varphi}_{l, \lambda}}}\eta_l)^2}{4}Z_{l,\lambda}^2\right]+O(Z_{l,\lambda}^3) \nonumber\\
 &=& -\mbox{Re}Z_{l,\lambda}+\frac{\mbox{Im}Z_{l,\lambda}^2}{4}+\eta_l \mbox{ terms}+O(Z_{l,\lambda}^3). \label{eq.2.10-3}
\end{eqnarray}
 Assume now $\lambda=\lambda_n=o(\sqrt n)$. Note that the interval considered in our context depends on $n$ and its length  tends to infinity as $n$.  But the basic estimates given by Proposition 22 of \cite{VV2009} for single-step asymptotics for $\varphi_{l, \lambda}$ still hold. Specifically speaking, for $l\le n_0$, we have from (\ref{eq.2.10-2}) and (\ref{eq.2.10-3})
\begin{eqnarray}
  E[\Delta {\varphi}_{l, \lambda}|{\varphi}_{l, \lambda}=x]=\frac{1}{n_0}b_{ l} +\frac{1}{n_0}osc_{1,l}+O\left(\frac{1}{(n_0-l)^{3/2}}\right),\label{eq.2.13}
  \end{eqnarray}
\begin{eqnarray}
  E[(\Delta {\varphi}_{l, \lambda} )^2|{\varphi}_{l, \lambda}=x ]&=&\frac{1}{n_0}a_{ l}  +\frac{1}{n_0}osc_{2,l}+O\left(\frac{1}{(n_0-l)^{3/2}}\right), \label{eq.2.14}
  \end{eqnarray}
  and
\begin{eqnarray}
   E[|\Delta {\varphi}_{l, \lambda}|^d|{\varphi}_{l, \lambda}]= O\left(\frac{1}{(n_0-l)^{3/2}}\right),\quad d\ge 3,\label{eq.2.15}
  \end{eqnarray}
where we use $osc_{1,l}$ and $osc_{2,l}$ to emphasize the dependence on $l$, and
$$
b_l=\frac{\sqrt {n_0}\lambda }{2\sqrt{n_0-l}}-n_0\frac{\mbox{Re}(\rho_{l+1}-\rho_l)}{\mbox{Im}\rho_l}+\frac{n_0\mbox{Im}(\rho_l^2)}{2\beta \sqrt{n_0-l}},$$
$$
a_l =\frac{2n_0}{\beta(n_0-l)} +\frac{n_0(3+\mbox{Re}\rho_l^2)}{\beta(n_0-l) }.
$$
The oscillatory terms are
$$
osc_{1,l}=\mbox{Re}\Bigl(\Bigl(-v_{l,\lambda } -i\frac {q_l}2\Bigr)e^{-ix}\eta_l\Bigr)+\frac 14\mbox{Re} \Bigl(ie^{-2ix}\eta^2_lq_l\Bigr)
$$
and
$$
osc_{1,2}=  {p_l} \mbox{Re}(e^{-ix}\eta_l )+ \mbox{Re}\Bigl[q_l\Bigl(e^{-ix}\eta_l +\frac 12e^{-i2x}\eta^2_l\Bigl)\Bigr],
$$
where
$$
p_l=\frac{4n_0}{\beta (n_0-l)},\quad q_l=\frac{2n_0(1+\rho_l^2)}{\beta (n_0-l)}.
$$

\section{Proofs of main results}

\noindent
{\bf Proof of Theorem \ref{th.1.1}} Take $l=\lfloor n(1-\frac{x^2}{4})-\frac 12 ((x^2n)^{1/3})\vee 1) \rfloor<n_0$. Then according to ($ v^{\prime}$), we have
\begin{eqnarray}
N_n\left(x\sqrt n,\quad x\sqrt n+\frac{t_n}{\sqrt n}\right)=\sharp[(\varphi_{l, 0}-\varphi_{l, 0}^\odot,\quad \varphi_{l, \lambda_n'}-\varphi_{l, \lambda_n'}^\odot)\cap 2\pi \mathbb{Z}],\nn\label{eq.3.1}
\end{eqnarray}
where $\lambda_n'=\frac{2t_n \sqrt{n_0}}{\sqrt n}$.

Since $t_n\rightarrow\infty$, it suffices to prove
\begin{eqnarray}
 \frac{1}{2\pi t_n}(\varphi_{l, \lambda_n'}-\varphi_{l, \lambda_n'}^\odot-(\varphi_{l, 0}-\varphi_{l, 0}^\odot))\stackrel P\longrightarrow \rho_{sc}(x).\nn\label{eq.3.2}
\end{eqnarray}
Note $\frac{ 2\sqrt{n_0}}{\sqrt n}\rightarrow \sqrt{4-x^2}$ by definition of $n_0$. We need only to prove
\begin{eqnarray}
 \frac{1}{t_n}(\varphi_{l, t_n }-\varphi_{l, t_n}^\odot-(\varphi_{l, 0}-\varphi_{l, 0}^\odot))\stackrel P\longrightarrow 1. \label{eq.3.3}
\end{eqnarray}
by a change of variable.

To deal with the term $\varphi_{l, t_n}^\odot-\varphi_{l, 0}^\odot$, we need the following lemma due to Valk$\acute{\mbox o}$ and Vira$\acute{\mbox g}$.
\begin{lemma}\label{lm3.1}
 Assume $t_n\rightarrow\infty$ such that $\frac{t_n}{\sqrt n}\rightarrow 0$. Then it follows
 $$
 \varphi_{l, t_n}^\odot-\varphi_{l, 0}^\odot\stackrel P\longrightarrow 0.
 $$
\end{lemma}
\begin{proof}
 It  is very similar to that of Lemma 34 in \cite{VV2009} with minor changes.
\end{proof}
To prove (\ref{eq.3.3}), it now remains to proving
\begin{eqnarray}
 \frac{\varphi_{l, t_n } }{t_n}\stackrel P\longrightarrow 1,  \quad \frac{\varphi_{l, 0} }{t_n}\stackrel P\longrightarrow 0. \label{eq.3.4}
\end{eqnarray}
In turn, (\ref{eq.3.4}) easily follows from
\begin{lemma}\label{lm3.2}
\begin{eqnarray}
\frac{1}{t_n}(E \varphi_{l, t_n } -t_n )\longrightarrow 0, \quad \frac{1}{t^2_n} Var(\varphi_{l, t_n } ) \longrightarrow 0\label{eq.3.5}
\end{eqnarray}
and
\begin{eqnarray}
\frac{1}{t_n} E\varphi_{l, 0} \longrightarrow 0,\quad \frac{1}{t^2_n}Var(\varphi_{l, 0})  \longrightarrow 0.  \label{eq.3.6}
\end{eqnarray}
 \end{lemma}
 \begin{proof} We shall only prove (\ref{eq.3.5}), since the other is very similar and simpler. First, note
 \begin{eqnarray}
 E \varphi_{l, t_n }  &=& E\sum_{k=0}^{l-1} \Delta\varphi_{k, t_n }\\
 &=& E\sum_{k=0}^{l-1} E ( \Delta\varphi_{k, t_n }|\varphi_{k,t_n}), \nn \label{eq.3.8}
\end{eqnarray}
where $\Delta\varphi_{k, \lambda }=\varphi_{k+1, \lambda }-\varphi_{k, \lambda}$ for any $\lambda$.
By virtue of the asymptotic estimate (\ref{eq.2.13}) for increments, we have
\begin{eqnarray}
E (\Delta\varphi_{k, t_n }|\varphi_{k,t_n})=\frac{1}{n_0}b_k+ osc_{1,k}+O\left(\frac{1}{(n_0-k)^{3/2}}\right),\label{eq.3.9}
\end{eqnarray}
where
$$
b_k=\frac{\sqrt{n_0}t_n}{2\sqrt{n_0-k}}- n_0\frac{\mbox{Re} (\rho_{k+1}-\rho_{k })}{\mbox{Im}\rho_k}+\frac{n_0\mbox{Im}(\rho_k^2)}{2\beta  {(n_0-k)}}
$$
and
$$
osc_{1,k}=\mbox{Re}((-n_0v_{k,t_n } -i\frac {q_k}2)e^{-i\varphi_{k,t_n}}\eta_k)+\frac 14\mbox{Re} (iq_ke^{-2i\varphi_{k,t_n}}\eta^2_k ).
$$
A simple calculus yields
\begin{eqnarray}
\frac{1}{n_0}\sum_{k=0}^{l-1}\frac{\sqrt{n_0}t_n}{2\sqrt{n_0-k}}= t_n(1+ o(1))\label{eq.3.10}
\end{eqnarray}
and
\begin{eqnarray}
\frac{1}{n_0}\sum_{k=0}^{l-1}\frac{\mbox{Re}(\rho_{k+1}-\rho_{k }) }{\mbox{Im}\rho_k}=o(t_n), \quad \frac{1}{n_0}\sum_{k=0}^{l-1}\frac{n_0\mbox{Im}(\rho_k^2)}{2\beta  {(n_0-k)}}=o(t_n).\label{eq.3.11}
\end{eqnarray}
Also, it follows from Lemma 37 of \cite{VV2009} that for any $x$
\begin{eqnarray}
\sum_{k=0}^{l-1} \mbox{Re}\left(\left(-n_0v_{k,\lambda } -i\frac {q_k}2\right)e^{-ix}\eta_k\right) =O(1),\quad  \sum_{k=0}^{l-1}\mbox{Re} (iqe^{-2ix}\eta^2_l )=O(1). \label{eq.3.12-1}
\end{eqnarray}
Combining (\ref{eq.3.9})-(\ref{eq.3.12-1}), we have
$$
\sum_{k=0}^{l-1} E (\Delta\varphi_{k, t_n }|\varphi_{k,t_n}  )=t_n(1+o(1)).
$$
 Thus we have shown the first statement of (\ref{eq.3.5}).

For the asymptotics for variance, note
\begin{eqnarray}
& &Var(\varphi_{l, t_n }) \nn \\& &=E\left(\sum_{k=0}^{l-1}\Delta\varphi_{k, t_n}-E\Delta\varphi_{k, t_n}\right)^2\nonumber\\
& &=E\left(\sum_{k=0}^{l-1}\Delta\varphi_{k, t_n}-E(\Delta\varphi_{k, t_n}|\varphi_{k,t_n})-(E\Delta\varphi_{k, t_n}-E(\Delta\varphi_{k, t_n}|\varphi_{k,t_n}))\right)^2.\nn\label{eq.3.10-1}
\end{eqnarray}
Since
\begin{eqnarray}
\sum_{k=0}^{l-1}(E\Delta\varphi_{k, t_n}-E(\Delta\varphi_{k, t_n}|\varphi_{k,t_n}))=O(1), \nn\label{eq.3.11-1}
\end{eqnarray}
it suffices to prove
\begin{eqnarray}
\frac 1{t^2_n}E\left(\sum_{k=0}^{l-1}\Delta\varphi_{k, t_n}-E(\Delta\varphi_{k, t_n}|\varphi_{k,t_n})  \right)^2\longrightarrow 0. \label{eq.3.12}
\end{eqnarray}
To this end, note $\varphi_{0, t_n}, \varphi_{1, t_n},\cdots, \varphi_{l,t_n}$ constructs a Markov chain so that $\Delta\varphi_{k, t_n}-E(\Delta\varphi_{k, t_n}|\varphi_{k,t_n} ), 1\le k\le l$ forms a martingale difference sequence. Hence it follows
\begin{eqnarray}
 E\left(\sum_{k=0}^{l-1}\Delta\varphi_{k, t_n}-E(\Delta\varphi_{k, t_n}|\varphi_{k,t_n}) \right)^2&=&\sum_{k=0}^{l-1} E \left(\Delta\varphi_{k, t_n}-E(\Delta\varphi_{k, t_n}|\varphi_{k,t_n})\right)^2\nonumber\\
 &\le& \sum_{k=0}^{l-1}  E (\Delta\varphi_{k, t_n})^2\nonumber\\
 &=&  E\sum_{k=0}^{l-1}  E \left((\Delta\varphi_{k, t_n})^2|\varphi_{k, t_n}\right).\label{eq.3.13}
\end{eqnarray}
Each conditional expectation in (\ref{eq.3.13}) is estimated by (\ref{eq.2.14}), and  note
\begin{eqnarray}
\frac{1}{n_0}\sum_{k=0}^{l-1} a_k&=&  \sum_{k=0}^{l-1}\frac{2 }{\beta(n_0-k)} +\frac{ 3+\mbox{Re}\rho_k^2 }{\beta(n_0-k) }\\
&=&
O(\log n) \nonumber \label{eq.3.15}
\end{eqnarray}
and
\begin{eqnarray}
\frac{1}{n_0}\sum_{k=0}^{l-1} Eosc_{2,k} =O(1).   \label{eq.3.16}
\end{eqnarray}
Therefore, (\ref{eq.3.12}) holds  since $\log n=o(t^2_n)$ by the assumption, which concludes the proof of Lemma \ref{lm3.2}.
\end{proof}

\noindent
{\bf Proof of Theorem \ref{th.1.2}} The proof is again based on the representation in ($v^\prime$) and the central limit theorem for Markov chain. Note   $x=0$ and  $n_0=n-\frac 12$. Take $l=  \lfloor n_0\rfloor=n-1 $, we have
\begin{eqnarray}
N_n(0, \infty)=\sharp ((\varphi_{n-1,0}-\varphi^\odot_{n-1,0},\quad \varphi_{n-1,\infty}-\varphi^\odot_{n-1,\infty})\cap 2\pi \mathbb{Z})\label{eq.3.17},
\end{eqnarray}
from which it readily follows
\begin{eqnarray}
\left|N_n(0, \infty)-\frac {1}{2\pi}(\varphi_{n-1,\infty}-\varphi^\odot_{n-1,\infty}-(\varphi_{n-1,0}-\varphi^\odot_{n-1,0} ))\right|\le 1. \label{eq.3.18}
\end{eqnarray}
We need the following lemma to compute the difference $\varphi_{n-1,\infty}-\varphi^\odot_{n-1,\infty}$ and $\varphi^\odot_{n-1,0}$.
\begin{lemma} \label{lm3.3} We have for $0<l\le n_0$

\noindent
(i)
 $$
\hat{\phi}_{l, \infty}=\pi, \qquad \hat{\phi}^{\odot}_{l, \infty}=-2(n-l)\pi;
$$
(ii)
$$
 {\phi}_{l, \infty}=(l+1)\pi,\qquad  \phi^{\odot}_{l, \infty}=-2n\pi+3l\pi;
$$
(iii)
$$
{\phi}^{\odot}_{l, 0}=\hat{\phi}^{\odot}_{l, 0}+l\pi, \qquad \hat{\phi}^{\odot}_{l, 0}=0_*\mathbf{R}^{-1}_{n-1,0}\cdots \mathbf{R}^{-1}_{l,0}.
$$
\end{lemma}
\begin{proof} First, prove $(i)$. Recall
 \begin{eqnarray}
\hat{\varphi}_{l, \infty}=\pi_*\mathbf{R}_{0, \infty}\cdots\mathbf{R}_{l-1, \infty}, \quad \hat{\varphi}^\odot_{l, \infty}=0_*\mathbf{R}_{n-1, \infty}^{-1}\cdots\mathbf{R}^{-1}_{l, \infty},\nn\label{eq.3.18-1}
\end{eqnarray}
where $\mathbf{R}_{l, \infty}=\mathbf{Q}(\pi)\mathbf{A}\left(1, \infty\right)\mathbf{W}_l$.

Note the affine transformation $\mathbf{A}\left(1, \infty\right)\mathbf{W}_l $ maps any $z$ to $\infty$, and the image of $\infty$ under the M$\ddot{\mbox o}$bius transformation $\mathbf{U}$ is $-1$, which in turn corresponds to $\pi\in \mathbb{R}^\prime $. Thus  it easily follows $ \hat{\phi}_{l, \infty}=\pi$.

For $\hat{\phi}^{\odot}_{l, \infty}$, note
\begin{eqnarray}
\hat{\phi}^{\odot}_{l, \infty}=\hat{\phi}^{\odot}_{l+1, \infty *} \mathbf{R}_{l, \infty}^{-1},\nn\label{eq.3.19}
\end{eqnarray}
where $\mathbf{R}_{l, \infty}^{-1}= \mathbf{W}_l^{-1}\mathbf{A}\left(1, -\infty\right)\mathbf{Q}(-\pi)$.

By the angular shift formula (see Fact 15 and (34) of \cite{VV2009}), we have
\begin{eqnarray}
\hat{\phi}^{\odot}_{l, \infty}&=&\hat{\phi}^{\odot}_{l+1, \infty *} \mathbf{W}_l^{-1}\mathbf{A}\left(1, -\infty\right)-\pi \nonumber\\
&=&\hat{\phi}^{\odot}_{l+1, \infty  }+ash \left( \mathbf{W}_l^{-1}\mathbf{A}\left(1, -\infty\right), -1, e^{i\hat{\phi}^{\odot}_{l+1, \infty  }}\right)-\pi\nonumber\\
&=&\hat{\phi}^{\odot}_{l+1, \infty  }+ \mbox{Arg}_{[0, 2\pi)}\left(\frac{e^{i\hat{\phi}^{\odot}_{l+1, \infty }}\hskip0.3mm_\circ \mathbf{W}_l^{-1}\mathbf{A}\left(1, -\infty\right) }{-1_\circ\mathbf{W}_l^{-1}\mathbf{A}\left(1, -\infty\right) }   \right)\nonumber\\
& &-\mbox{Arg}_{[0, 2\pi)}\left(\frac{e^{i\hat{\phi}^{\odot}_{l+1, \infty } } }{-1 }   \right)-\pi\nonumber\\
&=&\hat{\phi}^{\odot}_{l+1, \infty  }+ \mbox{Arg}_{[0, 2\pi)}(1)- \mbox{Arg}_{[0, 2\pi)}(-e^{i\hat{\phi}^{\odot}_{l+1, \infty } })-\pi,\nn
\label{eq.3.20}
\end{eqnarray}
from which and the fact $\hat{\phi}^{\odot}_{n, \infty}=0$  one can easily derive
\begin{eqnarray}
   \hat{\phi}^{\odot}_{n-1, \infty}=-2\pi, \quad \hat{\phi}^{\odot}_{n-1, \infty}=-4\pi,\cdots, \hat{\phi}^{\odot}_{l, \infty}= -2(n-l)\pi. \nn\label{eq.3.21}
\end{eqnarray}
Next we turn to $(ii)$ and $(iii)$. Since $x=0$, then $\rho_{l }=i$, and so $\mathbf{T}_l=\mbox{Id}$ and $\mathbf{Q}_{l-1}=\mathbf{Q}( l\pi)$ for each $0<l\le n_0$. Thus we have by (\ref{eq.2.8})
\begin{eqnarray}
   \varphi_{l,\lambda}&=&  \hat{\varphi}_{l,\lambda *}\mathbf{T}_l\mathbf{Q}_{l-1}\nn\\
   &=&  \hat{\varphi}_{l,\lambda *} \mathbf{Q}(l\pi)\nonumber\\
   &=&\hat{\varphi}_{l,\lambda }+l\pi, \nonumber \label{eq.3.22}
\end{eqnarray}
 where $0\le \lambda\le \infty$. Similarly, $\varphi^\odot_{l,\lambda}=\hat{\varphi}^\odot_{l,\lambda }+l\pi$.
\end{proof}
We now  apply  Lemma \ref{lm3.3}  to  immediately yield
\begin{eqnarray}
\varphi_{n-1,\infty}-\varphi^\odot_{n-1,\infty}=3\pi\label{eq.3.22-1}
\end{eqnarray}
and
\begin{eqnarray}
\varphi^\odot_{n-1,0}= 0_*\mathbf{R}^{-1}_{n-1,0}+ (n-1)\pi.\label{eq.3.22-2}
\end{eqnarray}
Also, it follows
\begin{eqnarray}
\frac{0_*\mathbf{R}^{-1}_{n-1,0}}{\sqrt{\log n}}\stackrel P\longrightarrow 0\label{eq.3.22-3}
\end{eqnarray}
 Combining  (\ref{eq.3.17})- (\ref{eq.3.22-3}) together, we need only to prove
\begin{eqnarray}
\frac{\varphi_{n-1, 0} }{\sqrt{\log n}}\stackrel d\longrightarrow \mathcal{N}\left(0,\frac{4}{\beta}\right).\label{eq.3.23}
\end{eqnarray}
To this end, we shall use the following central limit theorem for Markov chain. As mentioned in the proof of Theorem \ref{th.1.1}, $\pi=\varphi_{0, 0}, \varphi_{1, 0}, \cdots, \varphi_{n-1, 0}$ forms a Markov chain.  Let
$$
z_{l+1}=\Delta\varphi_{l, 0} -E(\Delta\varphi_{l, 0}|\varphi_{l, 0}).
$$
Then $z_1, z_2, \cdots, z_{n-1}$ forms a martingale difference sequence.  A classical martingale central limit theorem implies that if the following three conditions
 \begin{eqnarray}
s^2_n:= \sum_{l=1}^{n-1}Ez_l^2\rightarrow\infty,\label{eq.3.24}
\end{eqnarray}
 \begin{eqnarray}
 \frac{1}{s_n^2}\sum_{l=1}^{n-1}E\left(z_l^2| \varphi_{l-1, 0}\right)\stackrel P\longrightarrow 1\label{eq.3.25}
\end{eqnarray}
and
\begin{eqnarray}
 \frac{1}{s_n^2}\sum_{l=1}^{n-1}E\left(z_l^2\mathbf{1}_{(|z_l|>\varepsilon s_n)}| \varphi_{l-1, 0}\right)\stackrel P\longrightarrow 0, \quad \forall \varepsilon>0\label{eq.3.26}
\end{eqnarray}
 are satisfied, then we have
\begin{eqnarray}
  \frac{1}{s_n} \sum_{l=1}^{n-1}z_l\stackrel d\longrightarrow \mathcal{N}(0,1).\label{eq.3.27}
\end{eqnarray}
We next verify conditions (\ref{eq.3.24}) - (\ref{eq.3.26}) by asymptotic estimates (\ref{eq.2.13}) - (\ref{eq.2.15}) for the increments $E(\Delta\varphi_{l, 0} |\varphi_{l, 0})$. Let us begin with estimating $s^2_n$. Note
\begin{eqnarray}
  E(\Delta\varphi_{l, 0} |\varphi_{l, 0})= O\left(\frac{1}{(n_0-l)^{3/2}}\right)\nn\label{eq.3.28}
\end{eqnarray}
and
\begin{eqnarray}
  E(\Delta\varphi_{l, 0})^2&=&E(E((\Delta\varphi_{l, 0})^2 |\varphi_{l, 0}))\nn\\
  &=& \frac{4}{\beta(n_0-l)}  + \frac{4}{\beta(n_0-l)}E \mbox{Re}((-1)^{l+1}e^{-i\varphi_{l, 0}}) \nonumber\\
  & &+O\left(\frac{1}{(n_0-l)^{3/2}}\right).\nonumber\label{eq.3.29}
\end{eqnarray}
Hence a simple calculus show
\begin{eqnarray}
  s_n^2&=& \sum_{l=1}^{n-1} E(\Delta\varphi_{l, 0})^2-E( E(\Delta\varphi_{l, 0} |\varphi_{l, 0}))^2\nn\\
  &=& \frac{4}{\beta}\log n+O(1)\rightarrow\infty
   \nn \label{eq.3.30}
\end{eqnarray}
and
\begin{eqnarray}
  &  &\frac{1}{s_n^2}\sum_{l=1}^{n-1}E\left(z_l^2| \varphi_{l-1, 0}\right)-1\\
  & &= \frac{1}{s_n^2}\sum_{l=1}^{n-1}\left(E (z_l^2| \varphi_{l-1, 0})- E z_l^2\right)\nonumber\\
  & &=\frac{1}{s_n^2}\sum_{l=1}^{n-1}\left(E ((\Delta\varphi_{l-1 , 0})^2| \varphi_{l-1, 0})- E (\Delta\varphi_{l-1 , 0})^2\right)\nonumber\\
  & & \quad+\frac{1}{s_n^2}\sum_{l=1}^{n-1}E\left(E ( \Delta\varphi_{l-1 , 0} | \varphi_{l-1, 0})\right)  ^2- \left(E ( \Delta\varphi_{l-1 , 0}  |\varphi_{l-1, 0}) \right)^2
  \nonumber\\
  & & \stackrel P\longrightarrow 0. \nn
  \nonumber
    \label{eq.3.30-1}
\end{eqnarray}
It also follows from (\ref{eq.2.15}) that
\begin{eqnarray}
 \sum_{l=1}^{n-1}E|\Delta\varphi_{l-1 , 0}|^3=O(1), \nn
\end{eqnarray}
which in turn immediately implies the martingale type Lindeberg condition (\ref{eq.3.26}).  Thus (\ref{eq.3.27}),  and so (\ref{eq.3.23}) holds. This concludes the proof of Theorem \ref{th.1.2}.

\vskip2mm
\noindent
{\bf Concluding remarks} To conclude the paper, we  remark that  in the Theorem \ref{th.1.1} we only proved the convergence to semicircle law in probability with the help of most basic Markov's inequality. It is natural  and interesting ro ask what the convergence rate is. In particular, for what errors $\varepsilon_n$ and $\delta_n$ it follows
\begin{eqnarray}
P\left(\left|\frac{N_n\left(x\sqrt n, x\sqrt n+\frac{t_n}{\sqrt n}\right)}{t_n}-\rho_{sc}(x)\right|>\varepsilon_n\right)\le \delta_n. \nn
\end{eqnarray}
In special cases $\beta=1,2$, even in the cases of general Wigner symmetric and Hermitian matrices, there have been a lot of research works around the optimal errors in the past few years.   See \cite{E2010} and references therein  for relevant survey.

We also remark that Theorem \ref{th.1.2} only concerned the number of positive states on the real line. It is expected that the Gaussian fluctuation hold  for other half interval like $[x\sqrt n, \infty)$, but we do not find a suitable proof. According to ($v'$),  we have a.s. for $0<l\le  \frac{(4- x^2)n}{4} -\frac{1}{2}$
$$
N_{n}\left(x\sqrt n, \quad \infty\right)=\sharp(({\varphi}_{l, 0}-{\varphi}_{l, 0}^\odot,\quad {\varphi}_{l, \infty}-{\varphi}_{l, \infty}^\odot]\cap 2\pi \mathbb{Z}).
$$
and
$$
\left|N_{n}\left(x\sqrt n, \quad \infty\right)-\frac{1}{2\pi}({\varphi}_{l, \infty}-{\varphi}_{l, \infty}^\odot-({\varphi}_{l, 0}-{\varphi}_{l, 0}^\odot))\right|\le 1.
$$
However, we lack a specific estimate like Lemma \ref{lm3.3} so that the classical central limit theorem for Markov chain (martingale) is not applicable.  This is left to the future job.

\end{document}